\documentclass[11pt, DIV10,a4paper]{article}
\usepackage{natbib}
\newcommand {\ctn}{\citet} % change to \citet if using natbib
\usepackage{graphicx,subfigure,amsmath,latexsym,amssymb}
\usepackage{float,epsfig,multirow,rotating,times}
\usepackage{upgreek,wrapfig}

\newcommand{\bzero}{\boldsymbol{0}}

\newtheorem{theorem}{Theorem}

\newtheorem{corollary}[theorem]{Corollary}

\newenvironment{proof}[1][Proof]{\textbf{#1.} }{\ \rule{0.5em}{0.5em}}

\numberwithin{equation}{section}
\numberwithin{algo}{section}
\numberwithin{table}{section}
\numberwithin{figure}{section}

\usepackage{setspace}
\usepackage[mathscr]{euscript}
\usepackage[margin=1in]{geometry}
\begin{document}
%\renewcommand\baselinestretch{1.4}

%\normalsize

\title{\vspace{-0.8in}
On Asymptotics Related to Classical Inference in Stochastic Differential Equations with Random Effects}
\author{Trisha Maitra and Sourabh Bhattacharya\thanks{
Trisha Maitra is a PhD student and Sourabh Bhattacharya 
is an Associate Professor in Interdisciplinary Statistical Research Unit, Indian Statistical
Institute, 203, B. T. Road, Kolkata 700108.
Corresponding e-mail: sourabh@isical.ac.in.}}
\date{\vspace{-0.5in}}
\maketitle%

\begin{abstract}
\ctn{Maud12} considered $n$ independent stochastic differential equations ($SDE$'s), where
in each case the drift term is associated with a random effect, the distribution of which depends upon
unknown parameters. Assuming the independent and identical ($iid$) situation the authors provide independent proofs of
weak consistency and asymptotic normality of the maximum likelihood estimators ($MLE$'s) of the hyper-parameters
of their random effects parameters.

In this article, as an alternative route to proving consistency and asymptotic normality 
in the $SDE$ set-up involving random effects, we verify the regularity conditions required by existing relevant theorems.  
In particular, this approach allowed us to prove strong consistency under weaker assumption.
But much more importantly, we further consider the independent, but non-identical set-up associated 
with the random effects based
$SDE$ framework, and prove asymptotic results associated with the $MLE$'s.  
%Once again, verification of the regularity conditions did not turn out to be prohibitively
%difficult.
\\[2mm]
 
\end{abstract}

\section{Introduction}
\label{sec:intro}

\ctn{Maud12} study mixed-effects stochastic differential equations ($SDE$'s) of the following form:
\begin{equation}
d X_i(t)=b(X_i(t),\phi_i)dt+\sigma(X_i(t))dW_i(t),\quad\mbox{with}\quad X_i(0)=x^i,~i=1,\ldots,n.
\label{eq:sde1}
\end{equation}
Here, for $i=1,\ldots,n$, the stochastic process $X_i(t)$ is assumed to be continuously observed on the time interval $[0,T_i]$
with $T_i>0$ known, and
$\{x^i;~i=1,\ldots,n\}$ are the known initial values of the $i$-th process. 
The processes $\{W_i(\cdot);~i=1,\ldots,n\}$ are independent standard Brownian motions, and 
$\{\phi_i;~i=1,\ldots,n\}$ are independently and identically distributed ($iid$) random variables with common distribution
$g(\varphi,\theta)d\nu(\varphi)$ (for all $\theta$, $g(\varphi,\theta)$ is a density with respect to a dominating measure
on $\mathbb R^d$, where $\mathbb R$ is the real line and $d$ is the dimension), which are independent
of the Brownian motions. Here $\theta\in\Omega\subset\mathbb R^p$ ($p\geq 2d$) is an unknown parameter to be estimated. 
The functions $b:\mathbb R\times\mathbb R^d\mapsto\mathbb R$ and $\sigma:\mathbb R\mapsto\mathbb R$ 
are the drift function and the diffusion coefficient, respectively,
both assumed to be known. 
\ctn{Maud12} impose regularity conditions that ensure existence of solutions of (\ref{eq:sde1}). We adopt their assumptions,
which are as follows.
\begin{itemize}
\item[(H1)]
\begin{enumerate}
\item[(i)] The function $(x,\varphi)\mapsto b(x,\varphi)$ is $C^1$ (differentiable with continuous first derivative)
on $\mathbb R\times\mathbb R^d$, and such that there exists $K>0$ so that $$b^2(x,\varphi)\leq K(1+x^2+|\varphi|^2),$$
for all $(x,\varphi)\in\mathbb R\times\mathbb R^d$.
\item[(ii)] The function $\sigma(\cdot)$ is $C^1$ on $\mathbb R$ and
$$\sigma^2(x)\leq K(1+x^2),$$
for all $x\in\mathbb R$.
\end{enumerate}
\item[(H2)] Let $X^{\varphi}_i$ be associated with the SDE of the form (\ref{eq:sde1}) with drift function $b(x,\varphi)$.
Also letting $Q^{x^i,T_i}_{\varphi}$ denote the joint distribution of $\left\{X^{\varphi}_i(t);~t\in [0,T_i]\right\}$, it is
assumed that for $i=1,\ldots,n$, and for all $\varphi,\varphi'$, the following holds:
\[ 
Q^{x^i,T_i}_{\varphi}\left(\int_0^{T_i}\frac{b^2\left(X^{\varphi}_i(t),\varphi'\right)}{\sigma^2(X^{\varphi}_i(t))}dt<\infty\right)
=1.
\]
\item[(H3)] For $f=\frac{\partial b}{\partial\varphi_j},~j=1,\ldots,d$, there exist $c>0$ and some $\gamma\geq 0$ such that
\[
\underset{\varphi\in\mathbb R^d}{\sup}\frac{\left|f(x,\varphi)\right|}{\sigma^2(x)}\leq c\left(1+|x|^{\gamma}\right).
\]
\end{itemize}

Statistically, the $i$-th process $X_i(\cdot)$ can be thought of as modelling the $i$-th individual 
and the corresponding random variable $\phi_i$ denotes the random effect of individual $i$. 
For statistical inference, we
follow \ctn{Maud12} who consider the special case where $b(x,\phi_i)=\phi_ib(x)$. We assume 
\begin{itemize}
\item[(H1$^\prime$)] 
\begin{enumerate}
\item[(i)] $b(\cdot)$ and
$\sigma(x)$ are $C^1$ on $\mathbb R$ %having linear growth.
satisfying $b^2(x)\leq K(1+x^2)$  
and $\sigma^2(x)\leq K(1+x^2)$
for all $x\in\mathbb R$, for some $K>0$. %Assuming that  $\phi_i\stackrel{iid}{\sim}N(\mu,\omega^2)$, 
\item[(ii)] Almost surely for each $i\geq 1$, 
\[
\int_0^{T_i}\frac{b^2(X_i(s))}{\sigma^2(X_i(s))}ds<\infty.
\]
\end{enumerate}
\end{itemize}
Under this assumption, (H3) is no longer required; see \ctn{Maud12}. 
Moreover, Proposition 1 of \ctn{Maud12} holds; in particular, if for $k\geq 1$, $E|\phi_i|^{2k}<\infty$,
then for all $T>0$,
\begin{equation}
\underset{t\in[0,T]}{\sup}~E\left[X_i(t)\right]^{2k}<\infty.
\label{eq:finite_sup}
\end{equation}
As in \ctn{Maud12} we assume that $\phi_i$ are normally distributed. Hence, (\ref{eq:finite_sup}) is
satisfied in our case.
\ctn{Maud12} show that the likelihood, depending upon $\theta$, admits a relatively simple form composed of  
the following sufficient statistics:
\begin{align}
&U_i=\int_0^{T_i}\frac{b(X_i(s))}{\sigma^2(X_i(s))}dX_i(s),\quad V_i=\int_0^{T_i}\frac{b^2(X_i(s))}{\sigma^2(X_i(s))}ds,
\quad i=1,\ldots,n.
\label{eq:sufficient}
\end{align}  
%Assumption (H1$^\prime$) of \ctn{Maud12} ensures that 
%$V_i<\infty$ for all $i$.
%\begin{equation}
%V_i<\infty
%\label{eq:V_finite1}
%\end{equation}
%almost surely for every $i\geq 1$.
%We additionally assume that
%\begin{itemize}
%\item[(H4)]
%\begin{equation}
%\frac{b^2(x)}{\sigma^2(x)}\leq K(1+x^p), 
%~\mbox{for some}~p\geq 1.
%\label{eq:V_moments_finite}
%\end{equation}
%\end{itemize}
%This assumption, along with (\ref{eq:finite_sup}) ensures that $V_i$ has moments of all orders.
%
%Adopting assumption (H4) of \ctn{Maud12}, we also assume that
%\begin{itemize}
%\item[(H5)]The function $b(\cdot)/\sigma(\cdot)$
%is not constant, and that, for $i\geq 1$, $(U_i,V_i)$ has a density $\varrho_i(u,v)$ with respect to the Lebesgue
%measure on $\mathbb R\times \mathbb R^+$, which is jointly continuous and positive on an open ball
%of $\mathbb R\times\mathbb R^+$, where $\mathbb R^+=(0,\infty)$. 
%In the $iid$ set-up, $\varrho_i=\varrho_1$
%for $i\geq 1$.
%\end{itemize}
%
The exact likelihood is given by
\begin{equation}
L(\theta)=\prod_{i=1}^n\lambda_i(X_i,\theta),
\label{eq:likelihood1}
\end{equation}
where
\begin{equation}
\lambda_i(X_i,\theta)=\int_{\mathbb R}g(\varphi,\theta)\exp\left(\varphi U_i-\frac{\varphi^2}{2}V_i\right)d\nu(\varphi).
\label{eq:likelihood2}
\end{equation}
Assuming that $g(\varphi,\theta)d\nu(\varphi)\equiv N\left(\mu,\omega^2\right)$, \ctn{Maud12} obtain the following form
of $\lambda_i(X_i,\theta)$:
\begin{equation}
\lambda_i(X_i,\theta)=\frac{1}{\left(1+\omega^2V_i\right)^{1/2}}\exp\left[-\frac{V_i}{2\left(1+\omega^2V_i\right)}
\left(\mu-\frac{U_i}{V_i}\right)^2\right]\exp\left(\frac{U^2_i}{2V_i}\right),
\label{eq:likelihood3}
\end{equation}
where $\theta=(\mu,\omega^2)\in\Omega\subset\mathbb R\times\mathbb R^+$. As in \ctn{Maud12}, here we assume that
\begin{itemize}
\item[(H2$^\prime$)] $\Omega$ is compact. %and the true value $\theta_0$ is in the interior of $\Omega$.
\end{itemize}
%
%In the rest of this article we shall assume that the parameter space $\Omega\subset\mathbb R\times\mathbb R^+$
%is compact.
%
\ctn{Maud12} consider $x^i=x$ and $T_i=T$ for $i=1,\ldots,n$, so that the set-up boils down to the $iid$ situation, and
investigate asymptotic properties of the $MLE$ of $\theta$, providing proofs of consistency and asymptotic
normality independently, without invoking the general results already existing in the literature. 
In this article, as an alternative, we prove asymptotic properties of the $MLE$ in this $SDE$ set-up
by verifying the regularity conditions of relevant theorems already existing in the literature.
Our approach allowed us to prove strong consistency of $MLE$, rather than weak consistency proved by \ctn{Maud12}.
Also, importantly, our approach does not require assumption (H4) of \ctn{Maud12} which required
$(U_1,V_1)$ to have density with respect to the Lebsegue measure on $\mathbb R\times\mathbb R^+$, which
must be jointly continuous and positive on an open ball of $\mathbb R\times\mathbb R^+)$.

Far more importantly, we consider 
the independent but non-identical case (we refer to the latter as non-$iid$), and prove consistency
and asymptotic normality of the $MLE$ in this set-up. 
%Identifiability of the likelihood in the $iid$ set-up follows from the proof of Proposition 7 of \ctn{Maud12}; 
%this continues to hold in the non-$iid$ set-up as well.
In what follows, in Section \ref{sec:consistency_iid} we investigate asymptotic properties of $MLE$ 
in the $iid$ context. In Section \ref{sec:consistency_non_iid} we investigate classical 
asymptotics in the non-$iid$ set-up.
We summarize our work and provide concluding remarks in Section \ref{sec:conclusion}.
%Bayesian asymptotics in both the $iid$ and the non-$iid$ contexts are considered in \ctn{Maitra14b}.

Notationally, ``$\stackrel{a.s.}{\rightarrow}$", ``$\stackrel{P}{\rightarrow}$" and ``$\stackrel{\mathcal L}{\rightarrow}$"
denote convergence ``almost surely", ``in probability" and ``in distribution", respectively.

\section{Consistency and asymptotic normality of $MLE$ in the $iid$ set-up}
\label{sec:consistency_iid}

\subsection{Strong consistency of $MLE$}
\label{subsec:MLE_consistency_iid}

Consistency of the $MLE$ under the $iid$ set-up can be verified by validating the regularity
conditions of the following theorem (Theorems 7.49 and 7.54 of \ctn{Schervish95}); for our purpose
we present the version for compact $\Omega$.
\begin{theorem}[\ctn{Schervish95}]
\label{theorem:theorem1}
Let $\{X_n\}_{n=1}^{\infty}$ be conditionally $iid$ given $\theta$ with density $f_1(x|\theta)$
with respect to a measure $\nu$ on a space $\left(\mathcal X^1,\mathcal B^1\right)$. Fix $\theta_0\in\Omega$, and define,
for each $M\subseteq\Omega$ and $x\in\mathcal X^1$,
\[
Z(M,x)=\inf_{\psi\in M}\log\frac{f_1(x|\theta_0)}{f_1(x|\psi)}.
\]
Assume that for each $\theta\neq\theta_0$, there is an open set $N_{\theta}$ such that $\theta\in N_{\theta}$ and
that $E_{\theta_0}Z(N_{\theta},X_i)> -\infty$. 
%If $\Omega$ is not compact, assume further that there is a compact $C$ such that $\theta_0\in C$ and
%$E_{\theta_0}Z(\Omega\backslash C,X_i)>-\infty$.
Also assume that $f_1(x|\cdot)$ is continuous at $\theta$ 
for every $\theta$, a.s. $[P_{\theta_0}]$. Then, if $\hat\theta_n$ is the $MLE$ of $\theta$ corresponding to $n$ observations, 
it holds that $\underset{n\rightarrow\infty}{\lim}~\hat\theta_n=\theta_0$, a.s. $[P_{\theta_0}]$.
\end{theorem}

\subsubsection{Verification of strong consistency of $MLE$ in our SDE set-up}
\label{subsubsec:MLE_consistency_iid}
%Following \ctn{Maud12} we assume that the parameter space $\Omega$ is compact.
To verify the conditions of Theorem \ref{theorem:theorem1} in our case, we note that for any $x$,
$f_1(x|\theta)=\lambda_1(x,\theta)=\lambda(x,\theta)$ given by (\ref{eq:likelihood3}), 
which is clearly continuous in $\theta$. Also, it follows from the proof of Proposition 7 of \ctn{Maud12} that
for every $\theta\neq\theta_0$,
\begin{align}
\log\frac{f_1(x|\theta_0)}{f_1(x|\theta)}
&=\frac{1}{2}\log\left(\frac{1+\omega^2V_1}{1+\omega^2_0V_1}\right)
+\frac{1}{2}\frac{(\omega^2_0-\omega^2)U^2_1}{(1+\omega^2V_1)(1+\omega^2_0V_1)}\notag\\
&\quad+\frac{\mu^2V_1}{2(1+\omega^2V_1)}-\frac{\mu U_1}{1+\omega^2V_1}
-\left(\frac{\mu^2_0V_1}{2(1+\omega^2_0V_1)}-\frac{\mu_0U_1}{1+\omega^2_0V_1}\right)\notag\\
&\geq -\frac{1}{2}\left\{\log\left(1+\frac{\omega^2}{\omega^2_0}\right)+\frac{|\omega^2-\omega^2_0|}{\omega^2}\right\}
-\frac{1}{2}|\omega^2_0-\omega^2|\left(\frac{U_1}{1+\omega^2_0V_1}\right)^2\left(1+\frac{\omega^2_0}{\omega^2}\right)\notag\\
&\quad -|\mu|\left\vert\frac{U_1}{1+\omega^2_0V_1}\right\vert\left(1+\frac{|\omega^2_0-\omega^2|}{\omega^2}\right)
-\left|\frac{\mu^2_0V_1}{2(1+\omega^2_0V_1)}\right|-\left|\frac{\mu_0U_1}{1+\omega^2_0V_1}\right|.
\label{eq:lower_bound1}
\end{align}
Taking $N_{\theta}=\left(\underline\mu,\overline\mu\right)\times \left(\underline\omega^2,\overline\omega^2\right)$, and
noting that $E_{\theta_0}\left(\frac{U_1}{1+\omega^2_0V_1}\right)^2$, 
$E_{\theta_0}\left\vert\frac{U_1}{1+\omega^2_0V_1}\right\vert$ and 
$E_{\theta_0}\left(\frac{\mu^2_0V_1}{2(1+\omega^2_0V_1)}\right)$
are finite due to Lemma 1 of \ctn{Maud12}, it follows that $E_{\theta_0}Z(N_{\theta},X_i)> -\infty$.
Hence, $\hat\theta_n\stackrel{a.s.}{\rightarrow}\theta_0$ $[P_{\theta_0}]$.
We summarize the result in the form of the following theorem:
\begin{theorem}
\label{theorem:consistency_iid}
Assume the $iid$ setup and conditions (H1$^\prime$) and (H2$^\prime$). %, (H4), (H5), (H6). 
Then the $MLE$ is strongly consistent
in the sense that
$\hat\theta_n\stackrel{a.s.}{\rightarrow}\theta_0$~ $[P_{\theta_0}]$.
\end{theorem}

\subsection{Asymptotic normality of $MLE$}
\label{subsec:MLE_iid_normality}

To verify asymptotic normality of $MLE$ we invoke the following theorem provided in \ctn{Schervish95} (Theorem 7.63):
\begin{theorem}[\ctn{Schervish95}]
\label{theorem:theorem2}
Let $\Omega$ be a subset of $\mathbb R^d$, and let $\{X_n\}_{n=1}^{\infty}$ be conditionally $iid$ given $\theta$
each with density $f_1(\cdot|\theta)$. Let $\hat\theta_n$ be an $MLE$. Assume that
$\hat\theta_n\stackrel{P}{\rightarrow}\theta$ under $P_{\theta}$ for all $\theta$. Assume that $f_1(x|\theta)$
has continuous second partial derivatives with respect to $\theta$ and that differentiation can be passed under the
integral sign. Assume that there exists $H_r(x,\theta)$ such that, for each $\theta_0\in int(\Omega)$ and each
$k,j$,
\begin{align}
\sup_{\|\theta-\theta_0\|\leq r}\left\vert\frac{\partial^2}{\partial\theta_k\partial\theta_j}\log f_{X_1|\Theta}(x|\theta_0)
-\frac{\partial^2}{\partial\theta_k\partial\theta_j}\log f_{X_1|\Theta}(x|\theta)\right\vert\leq H_r(x,\theta_0),
\label{eq:H1}
\end{align}
with
\begin{equation}
\lim_{r\rightarrow 0}E_{\theta_0}H_r\left(X,\theta_0\right)=0.
\label{eq:H2}
\end{equation}
Assume that the Fisher information matrix $\mathcal I(\theta)$ is finite and non-singular. Then, under $P_{\theta_0}$,
\begin{equation}
\sqrt{n}\left(\hat\theta_n-\theta_0\right)\stackrel{\mathcal L}{\rightarrow}N\left(\bzero,\mathcal I^{-1}(\theta_0)\right). 
\label{eq:MLE_normality_iid}
\end{equation}
\end{theorem}

\subsubsection{Verification of the above regularity conditions for asymptotic normality in our SDE set-up}
\label{subsubsec:MLE_normality_iid}
In Section \ref{subsubsec:MLE_consistency_iid} we proved almost sure consistency of the $MLE$ $\hat\theta_n$
in the SDE set-up.
Hence, $\hat\theta_n\stackrel{P}{\rightarrow}\theta$ under $P_{\theta}$ for all $\theta$.
In the proof of Proposition 5, \ctn{Maud12} show that differentiation can be passed under the integral sign.
Letting $\gamma_i(\theta)=\frac{U_i-\mu V_i}{1+\omega^2V_i}$ and $I_i=\frac{V_i}{1+\omega^2V_i}$, 
note that (see the proof of Proposition 6 of \ctn{Maud12})
\begin{align}
& \frac{\partial^2}{\partial\mu^2}\log f_1(x|\theta) = -I_1(\omega^2),\quad 
\frac{\partial^2}{\partial\mu\partial\omega^2}\log f_1(x|\theta) = -\gamma_1(\theta)I_1(\omega^2);\label{eq:diff1}\\
&\frac{\partial^2}{\partial\omega^2\partial\omega^2}\log f_1(x|\theta) 
= -\frac{1}{2}\left(2\gamma^2_1(\theta)I_1(\omega^2)-I^2_1(\omega^2)\right).\label{eq:diff2}
\end{align}
It follows from (\ref{eq:diff1}) and (\ref{eq:diff2}) that in our case 
$\frac{\partial^2}{\partial\theta_k\partial\theta_j}\log f_1(x|\theta)$ is differentiable in 
$\theta=(\mu,\omega^2)$, and the derivative has finite expectation; see the proof of Proposition 8 of
\ctn{Maud12}). Hence, (\ref{eq:H1}) and (\ref{eq:H2}) clearly hold.
Following \ctn{Maud12} we assume:
\begin{itemize}
\item[(H3$^\prime$)] The true value $\theta_0\in int\left(\Omega\right)$.
\end{itemize}
That the information matrix $\mathcal I(\theta)$ is finite and is the covariance matrix
of the vector $\left(\gamma_1(\theta),\frac{1}{2}\left(\gamma^2_1(\theta)-I_1\left(\omega^2\right)\right)\right)$ 
(hence, nonnegative-definite), are shown
in \ctn{Maud12}. We additionally assume, as \ctn{Maud12}:
\begin{itemize}
\item[(H4$^\prime$)] The information matrix $\mathcal I(\theta_0)$ is invertible.
\end{itemize}
Hence, asymptotic normality of the $MLE$, of the form (\ref{eq:MLE_normality_iid}), holds in our case.
Formally,
\begin{theorem}
\label{theorem:asymp_normal_iid}
Assume the $iid$ setup and conditions (H1$^\prime$) -- (H4$^\prime$). %, (H4), (H5), (H6). 
Then the $MLE$ is asymptotically normally distributed as
(\ref{eq:MLE_normality_iid}).
\end{theorem}

\section{Consistency and asymptotic normality of $MLE$ in the non-$iid$ set-up}
\label{sec:consistency_non_iid}

We now consider the case where the processes $X_i(\cdot);~i=1,\ldots,n$, are independently,
but not identically distributed. This happens when we no longer enforce the restrictions $T_i=T$
and $x^i=x$ for $i=1,\ldots,n$. However, we do assume that the sequences $\{T_1,T_2,\ldots\}$ and 
$\{x^1,x^2,\ldots,\}$ are sequences entirely contained in compact sets $\mathfrak T$ and $\mathfrak X$, respectively.
Due to compactness, there exist convergent subsequences with limits in $\mathfrak T$ and $\mathfrak X$.
%, respectively.
Abusing notation, we continue to denote the convergent subsequences as  $\{T_1,T_2,\ldots\}$
and $\{x^1,x^2,\ldots\}$. Let the limts be $T^{\infty}\in \mathfrak T$ and $x^{\infty}\in \mathfrak X$.
%, respectively. 

Now, since the distributions of the processes $X_i(\cdot)$ are uniquely defined on
the space of real, continuous functions $\mathcal C\left([0,T_i]\mapsto\mathbb R\right)=
\left\{f:[0,T_i]\mapsto\mathbb R~\mbox{such that}~f~\mbox{is continuous}\right\}$,
given any $t\in [0,T_i]$, $f(t)$ is clearly a continuous function of the initial value $f(0)=x$.
To emphasize dependence on $x$, we denote the function as $f(t,x)$.
In fact, for any $\epsilon>0$, there exists $\delta_{\epsilon}>0$ 
such that whenever $|x_1-x_2|<\delta_{\epsilon}$,  $|f(t,x_1)-f(t,x_2)|<\epsilon$ for all $t\in [0,T_i]$.
%Therefore, well-defined moments of all continuous functions of $X_i(\cdot)$ are continuous functions of $x$.

%Theorem \ref{theorem:continuity} shows that %in (\ref{eq:sufficient}), well-defined moments of 
%$U_i$ and $V_i$ are continuous in $T_i=T$ and $x_i=x$,
%provided that the functions $b(\cdot)$ and $\sigma^2(\cdot)$ are continuous, the latter being non-zero, which we assume for
%our purpose. For notational convenience in the present situation, let us emphasize the dependence of $X(t)$ on $x$ by denoting
%the process by $X(t,x)$. The proof of the theorem is provided in the Appendix.
%\begin{theorem}
%\label{theorem:continuity}

Henceforth, we denote the process associated with the initial value $x$ and time point $t$ as $X(t,x)$, and by $\phi(x)$ the random effect
parameter associated with the initial value $x$ such that $\phi(x^i)=\phi_i$. We assume that 
\begin{itemize}
\item[(H5$^\prime$)] $\phi(x)$ is a real-valued, continuous function of $x$, and that 
for $k\geq 1$, 
\begin{equation}
\underset{x\in \mathfrak X}{\sup}~E\left[\phi(x)\right]^{2k}<\infty. 
\label{eq:phi_sup1}
\end{equation}
\end{itemize}
%As in Proposition 1 of \ctn{Maud12}, this implies that for any $T>0$,
%\begin{equation}
%\underset{t\in [0,T],x\in \mathfrak X}{\sup}~E\left[X(t,x)\right]^{2k}<\infty.
%\label{eq:sup_X}
%\end{equation}

For $x\in \mathfrak X$ and $T\in \mathfrak T$, let
\begin{align}
U(x,T)&=\int_0^T\frac{b(X(s,x))}{\sigma^2(X(s,x))}d X(s,x)\label{eq:u_x_T};\\
V(x,T)&=\int_0^T\frac{b^2(X(s,x))}{\sigma^2(X(s,x))}ds.\label{eq:v_x_T}
\end{align}
Clearly, $U(x^i,T_i)=U_i$ and $V(x^i,T_i)=V_i$, where $U_i$ and $V_i$ are given by
(\ref{eq:sufficient}).
%Analogous to assumption (H1$^\prime$) here we assume that
%\begin{itemize}
%\item[(H6$^\prime$)]
%\begin{equation}
%V(x,T)<\infty
%\label{eq:V_finite2}
%\end{equation}
%almost surely for every $x\in\mathfrak X$ and $T\in\mathfrak T$.
%\end{itemize}
In this non-$iid$ set-up we assume that
\begin{itemize}
\item[(H6$^\prime$)]
\begin{equation}
\frac{b^2(x)}{\sigma^2(x)}<K(1+x^\tau),~\mbox{for some}~\tau\geq 1.
\label{eq:V_finite_moment2}
\end{equation}
\end{itemize}
This assumption ensures that moments of all orders of $V(x,T)$ are finite.
Then the moments of uniformly integrable continuous functions of $U(x,T)$, $V(x,T)$ and $\theta$ are continuous in 
$x$, $T$ and $\theta$.
The result is formalized as Theorem \ref{theorem:moment_continuity}, the proof of which is presented
in the Appendix.
\begin{theorem}
\label{theorem:moment_continuity}
Assume  (H5$^\prime$) and (H6$^\prime$). Let $h(u,v,\theta)$ be any continuous function of $u$, $v$ and $\theta$,
such that for any sequences $\left\{x_m\right\}_{m=1}^{\infty}$, 
$\left\{T_m\right\}_{m=1}^{\infty}$ and $\left\{\theta_m\right\}_{m=1}^{\infty}$, 
converging to $\tilde x$, $\tilde T$ and $\tilde\theta$, respectively, for any $\tilde x\in \mathfrak X$, 
$\tilde T\in \mathfrak T$ and $\tilde\theta\in\Omega$, 
the sequence $\left\{h\left(U(x_m,T_m),V(x_m,T_m),\theta_m\right)\right\}_{m=1}^{\infty}$
is uniformly integrable.
Then, as $m\rightarrow\infty$, %$E\left[h\left(U(x,T),V(x,T)\right)\right]$ is continuous in $x$ and $T$ in the sense that
\begin{equation}
E\left[h\left(U(x_m,T_m),V(x_m,T_m),\theta_m\right)\right]
\rightarrow E\left[h\left(U(\tilde x,\tilde T),V(\tilde x,\tilde T),\tilde\theta\right)\right].
\label{eq:moment1}
\end{equation}
\end{theorem}

\begin{corollary}
\label{corollary:corollary1}
As in \ctn{Maud12}, consider the function
\begin{align}
%h_1(u,v)&=\frac{v}{1+\omega^2_0v}\label{eq:h1};\\
h(u,v)&=\exp\left(\psi\frac{u}{1+\xi v}\right)\label{eq:h2},
\end{align}
where $\psi\in\mathbb R$ and $\xi\in \mathbb R^+$.
Then, for any sequences $\left\{x_m\right\}_{m=1}^{\infty}$ and 
$\left\{T_m\right\}_{m=1}^{\infty}$ converging to $\tilde x$ and $\tilde T$, for any $\tilde x\in \mathfrak X$ and
$\tilde T\in \mathfrak T$, and for $k\geq 1$, 
\begin{equation}
E\left[h\left(U(x_m,T_m),V(x_m,T_m)\right)\right]^k
\rightarrow E\left[h\left(U(\tilde x,\tilde T),V(\tilde x,\tilde T)\right)\right]^k,
\label{eq:moment2}
\end{equation}
as $m\rightarrow\infty$.
\end{corollary}
The proof of the above corollary only entails proving uniform integrability of 
$\left\{h\left(U(x_m,T_m),V(x_m,T_m)\right)\right\}_{m=1}^{\infty}$, which simply follows from the %(H6$^\prime$) and the
proof of Lemma 1 of \ctn{Maud12}.

Note that in our case, the Kullback-Leibler distance and Fisher's information are expectations of 
functions of the form $h(u,v,\theta)$, continuous in $u$, $v$ and $\theta$. Assumption (H6$^\prime$), 
the upper bounds provided
in \ctn{Maud12}, Corollary \ref{corollary:corollary1}, and compactness of $\Omega$, can be used to easily 
verify uniform integrability of the relevant sequences.
It follows that in our situation the Kullback-Leibler distance, which we now denote by $\mathcal K_{x,T}(\theta_0,\theta)$
(or $\mathcal K_{x,T}(\theta,\theta_0)$) to emphasize dependence on $x$, $T$ and $\theta$ are continuous in $\theta$, $x$
and $T$. Similarly, the elements of the Fisher's information matrix $\mathcal I_{x,T}(\theta)$ are continuous
in $\theta$, $x$ and $T$. For $x=x^k$ and $T=T_k$, we denote the Kullback-Leibler 
distance and the Fisher's information
as $\mathcal K_k(\theta_0,\theta)$ ($\mathcal K_k(\theta,\theta_0)$) and $\mathcal I_k(\theta)$, respectively.

Continuity of $\mathcal K_{x,T}(\theta_0,\theta)$ (or $\mathcal K_{x,T}(\theta,\theta_0)$) and $\mathcal I_{x,T}(\theta_0)$
with respect to $x$ and $T$ ensures that as $x^k\rightarrow x^{\infty}$ and $T_k\rightarrow T^{\infty}$,
$\mathcal K_{x^k,T_k}(\theta_0,\theta)\rightarrow \mathcal K_{x^{\infty},T^{\infty}}(\theta_0,\theta)=\mathcal K(\theta_0,\theta)$, say.
Similarly, $\mathcal K_{x^k,T_k}(\theta,\theta_0)\rightarrow \mathcal K(\theta,\theta_0)$ and 
$\mathcal I_{x^k,T_k}(\theta)\rightarrow \mathcal I_{x^{\infty},T^{\infty}}(\theta)=
\mathcal I(\theta)$, say. Since $X^{\infty}$ and $T^{\infty}$ are contained in the respective compact sets,
the limits $\mathcal K(\theta_0,\theta)$, $\mathcal K(\theta,\theta_0)$ and $\mathcal I(\theta)$ are well-defined
Kullback-Leibler divergences and Fisher's information, respectively.
From the above limits, it follow that for any $\theta\in\Omega$,
\begin{align}
\underset{n\rightarrow\infty}{\lim}~\frac{\sum_{k=1}^n\mathcal K_k(\theta_0,\theta)}{n}&=\mathcal K(\theta_0,\theta);
\label{eq:kl_limit_1}\\
\underset{n\rightarrow\infty}{\lim}~\frac{\sum_{k=1}^n\mathcal K_k(\theta,\theta_0)}{n}&=\mathcal K(\theta,\theta_0);
\label{eq:kl_limit_2}\\
\underset{n\rightarrow\infty}{\lim}~\frac{\sum_{k=1}^n\mathcal I_k(\theta)}{n}&=\mathcal I(\theta).
\label{eq:fisher_limit_1}
\end{align}
We investigate consistency and asymptotic normality of $MLE$ in our case using the results 
of \ctn{Hoadley71}. 
%For consistency in the Bayesian framework we utilize the theorem 
%of \ctn{Choi04}, and for asymptotic normality of the posterior we make use of Theorem 7.89
%of \ctn{Schervish95}.
The limit results (\ref{eq:kl_limit_1}), (\ref{eq:kl_limit_2}) and (\ref{eq:fisher_limit_1}) will play
important roles in our proceedings.

\subsection{Consistency and asymptotic normality of $MLE$ in the non-$iid$ set-up}
\label{subsec:consistency_non_iid}

Following \ctn{Hoadley71} we define the following:
\begin{align}
R_i(\theta)&=\log\frac{f_i(X_i|\theta)}{f_i(X_i|\theta_0)}\quad\mbox{if}\ \ f_i(X_i|\theta_0)>0\notag\\
&=0 \quad\quad\quad\quad\quad\quad\mbox{otherwise}.
\label{eq:R1}
\end{align}

\begin{align}
R_i(\theta,\rho)&=\sup\left\{R_i(\xi):\|\xi-\theta\|\leq\rho\right\}\label{eq:R2}\\
{\mathcal V}_i(r)&=\sup\left\{R_i(\theta):\|\theta\|>r\right\}.\label{eq:V}
\end{align}
Following \ctn{Hoadley71} we denote by $r_i(\theta)$, $r_i(\theta,\rho)$ and $v_i(r)$
to be expectations of $R_i(\theta)$, $R_i(\theta,\rho)$ and ${\mathcal V}_i(r)$ under $\theta_0$; 
for any sequence $\{a_i;i=1,2,\ldots\}$ we denote
$\sum_{i=1}^na_i/n$ by $\bar a_n$.

\ctn{Hoadley71} proved that if the following regularity conditions are satisfied, then 
the MLE $\hat\theta_n\stackrel{P}{\rightarrow}\theta_0$:
\begin{itemize}
\item[(1)] $\Omega$ is a closed subset of $\mathbb R^d$.
\item[(2)] $f_i(X_i|\theta)$ is an upper semicontinuous %(u.s.c) 
function of $\theta$, uniformly in $i$, 
a.s. $[P_{\theta_0}]$.
\item[(3)] There exist $\rho^*=\rho^*(\theta)>0$, $r>0$ and $0<K^*<\infty$ for which 
\begin{enumerate}
\item[(i)] $E_{\theta_0}\left[R_i(\theta,\rho)\right]^2\leq K^*,\quad 0\leq\rho\leq\rho^*$;
\item[(ii)] $E_{\theta_0}\left[{\mathcal V}_i(r)\right]^2\leq K^*$.
\end{enumerate}
\item[(4)]
\begin{enumerate}
\item[(i)]$\underset{n\rightarrow\infty}{\lim}~\bar r_n(\theta)<0,\quad\theta\neq\theta_0$;
\item[(ii)]$\underset{n\rightarrow\infty}{\lim}~\bar v_n(r)<0$.
\end{enumerate}
\item[(5)] $R_i(\theta,\rho)$ and ${\mathcal V}_i(r)$ are measurable functions of $X_i$.
\end{itemize}
Actually, conditions (3) and (4) can be weakened but these are more easily applicable (see \ctn{Hoadley71} for details).

\subsubsection{Verification of the regularity conditions}
\label{subsubsec:consistency_non_iid}

Since $\Omega$ is compact in our case, the first regularity condition
clearly holds. 

For the second regularity condition, note that given $X_i$, 
$f_i(X_i|\theta)$ is continuous, in fact, uniformly continuous in $\theta$ in our case, 
since $\Omega$ is compact. Hence, for any given $\epsilon>0$, there exists $\delta_i(\epsilon)>0$, independent
of $\theta$,
such that $\|\theta_1-\theta_2\|<\delta_i(\epsilon)$ implies $\left|f(X_i|\theta_1)-f(X_i|\theta_2)\right|<\epsilon$.
Now consider a strictly positive function $\delta_{x,T}(\epsilon)$, continuous in $x\in\mathfrak X$ and $T\in\mathfrak T$,
such that $\delta_{x^i,T_i}(\epsilon)=\delta_i(\epsilon)$. Let 
$\delta(\epsilon)=\underset{x\in\mathfrak X,T\in\mathfrak T}{\inf}\delta_{x,T}(\epsilon)$. Since
$\mathfrak X$ and $\mathfrak T$ are compact, it follows that $\delta(\epsilon)>0$. Now it holds that
$\|\theta_1-\theta_2\|<\delta(\epsilon)$ implies $\left|f(X_i|\theta_1)-f(X_i|\theta_2)\right|<\epsilon$,
for all $i$. Hence, the second regularity condition is satisfied.
%For condition (3)(i), note that for $0\leq\rho\leq\rho^*(\theta)$, where
%$\rho^*(\theta)$ is chosen as small as desired,
%\begin{align}
%\sup\left\{R_i(\xi):\|\xi-\theta\|\leq\rho\right\} &= \sup\left\{R_i(\xi)-R_i(\theta)
%+R_i(\theta):\|\xi-\theta\|\leq\rho\right\}\notag\\
%&\leq\sup\left\{\left\vert R_i(\xi)-R_i(\theta)\right\vert:\|\xi-\theta\|\leq\rho\right\}+R_i(\theta)\notag\\
%&\leq\rho~\underset{\theta\in\Omega}{\sup}\|\nabla R_i(\theta)\|+%\underset{\theta\in\Omega}{\sup}
%R_i(\theta),
%\label{eq:cond1}
%\end{align}
%where $\nabla R_i(\theta)=\left(\gamma_i(\theta),\frac{1}{2}\left(\gamma^2_i(\theta)-I_i(\omega^2)\right)\right)$
%is the gradient of $R_i(\theta)$ and $\|\cdot\|$ is the Euclidean norm. 
%Hence,
%for $\Omega$ of the form $[\underline\mu,\overline\mu]\times [\underline\omega^2,\overline\omega^2]$, with
%$\underline\mu<\overline\mu$ and $0<\underline\omega^2<\overline\omega^2$,
%\begin{align}
%\|\nabla R_i(\theta)\|\leq |\gamma_i(\theta)|\sqrt{1+\gamma^2_i(\theta)},
%\label{eq:cond1_2}
%\end{align}
%where $\gamma_i(\theta)$ satisfies the following inequality (see \ctn{Maud12})
%\begin{equation}
%\underset{\theta\in\Omega}{\sup}|\gamma_i(\theta)|\leq\left\vert\frac{U_i}{1+\omega^2_0V_i}\right|
%\left(2+\frac{\omega^2_0}{\underline\omega^2}\right)+\frac{|\bar\mu |}{\underline\omega^2}.
%\label{eq:cond1_3}
%\end{equation}

Let us now focus attention on condition (3)(i).
It follows from (\ref{eq:lower_bound1}) that
\begin{align}
R_i(\theta) &\leq \frac{1}{2}\left\{\log\left(1+\frac{\omega^2}{\omega^2_0}\right)+\frac{|\omega^2-\omega^2_0|}{\omega^2}\right\}
+\frac{1}{2}|\omega^2_0-\omega^2|\left(\frac{U_i}{1+\omega^2_0V_i}\right)^2\left(1+\frac{\omega^2_0}{\omega^2}\right)\notag\\
&\quad +|\mu|\left\vert\frac{U_i}{1+\omega^2_0V_i}\right\vert\left(1+\frac{|\omega^2_0-\omega^2|}{\omega^2}\right)
+\left(\frac{\mu^2_0V_i}{2(1+\omega^2_0V_i)}-\frac{\mu_0U_i}{1+\omega^2_0V_i}\right).
%&=K_3(\theta),~\mbox{(say).}
\label{eq:upper_bound1}
\end{align}
Let us denote $\left\{\xi\in\mathbb R\times\mathbb R^+:\|\xi-\theta\|\leq\rho\right\}$ by 
$S(\rho,\theta)$. Here $0<\rho<\rho^*(\theta)$, and $\rho^*(\theta)$ 
is so small that
$S(\rho,\theta)\subset\Omega$ for all $\rho\in (0,\rho^*(\theta))$. It then follows from (\ref{eq:upper_bound1}) that
\begin{align}
\underset{\xi\in S(\rho,\theta)}{\sup}~R_i(\xi)
&\leq \underset{(\mu,\omega^2)\in S(\rho,\theta)}{\sup}~\frac{1}{2}\left\{\log\left(1+\frac{\omega^2}{\omega^2_0}\right)
+\frac{|\omega^2-\omega^2_0|}{\omega^2}\right\}\notag\\
&\quad+\left(\frac{U_i}{1+\omega^2_0V_i}\right)^2\times
\underset{(\mu,\omega^2)\in S(\rho,\theta)}{\sup}~\left[\frac{1}{2}\left|\omega^2_0-\omega^2\right|
\left(1+\frac{\omega^2_0}{\omega^2}\right)\right]\notag\\
&\quad+\left\vert\frac{U_i}{1+\omega^2_0V_i}\right\vert\times\underset{(\mu,\omega^2)\in S(\rho,\theta)}{\sup}~
\left[|\mu|\left(1+\frac{|\omega^2_0-\omega^2|}{\omega^2}\right)\right]\notag\\
&\quad +\left|\frac{\mu^2_0V_i}{2(1+\omega^2_0V_i)}\right|+\left|\frac{\mu_0U_i}{1+\omega^2_0V_i}\right|.
\label{eq:sup_R1}
\end{align}
The supremums in (\ref{eq:sup_R1}) are finite due to compactness of $S(\rho,\theta)$. 
Since under $P_{\theta_0}$, $U_i/(1+\omega^2_0V_i)$ admits moments of all orders and 
$0<I_i(\omega^2_0)=\frac{V_i}{1+\omega^2_0V_i}<\frac{1}{\omega^2_0}$ (see \ctn{Maud12}), it follows from
(\ref{eq:sup_R1}) that
%(\ref{eq:cond1}), (\ref{eq:cond1_2}), (\ref{eq:cond1_3}) and (\ref{eq:upper_bound1}) that  
\begin{equation}
E_{\theta_0}\left[R_i(\theta,\rho)\right]^2\leq K_i(\theta),
\label{eq:upper_bound2}
\end{equation}
where $K_i(\theta)=K(x^i,T_i,\theta)$, with $K(x,T,\theta)$ being a continuous function of 
$(x,T,\theta)$, continuity being a consequence
of Theorem \ref{theorem:moment_continuity}. 
Since 
because of compactness of $\mathfrak X$, $\mathfrak T$ and $\Omega$,
$$K_i(\theta)\leq \underset{x\in\mathfrak X,T\in\mathfrak T,\theta\in\Omega}{\sup}~K(x,T,\theta)<\infty,$$
%On taking supremum
%of $K(\theta)$ over $\Omega$, 
regularity condition (3)(i) follows.

To verify condition (3)(ii), first note that we can choose $r>0$ such that $\|\theta_0\|<r$ and
$\{\theta\in\Omega:\|\theta\|>r\}\neq\emptyset$. %$\subset\Omega$.
%Since $\Omega$ is compact, (3)(ii) follows easily from (\ref{eq:upper_bound1}) and due to existence
%of moments of all orders.
%Clearly, $R_i(\theta)=R_{x^i,T_i}(\theta)$. For any $r>0$ such that $\left\{\theta:\|\theta\|>r\right\}\subset\Omega$, 
It then follows that %from (\ref{eq:upper_bound1}) that 
$\underset{\left\{\theta\in\Omega:\|\theta\|>r\right\}}{\sup}~R_i(\theta)\leq \underset{\theta\in\Omega}{\sup}~R_i(\theta)$
for every $i\geq 1$. The right hand side is bounded by the same expression as the right hand side of (\ref{eq:sup_R1}),
with only $S(\theta,\rho)$ replaced with $\Omega$. %Thus, as in the case of (3)(i), here also, we obtain
The rest of the verification follows in the same way as verification of (3)(i).

To verify condition (4)(i) note that by (\ref{eq:kl_limit_1}) 
\begin{equation}
\underset{n\rightarrow\infty}{\lim}~\bar r_n=-\underset{n\rightarrow\infty}{\lim}~\frac{\sum_{i=1}^n\mathcal K_i(\theta_0,\theta)}{n}
=-\mathcal K(\theta_0,\theta)<0\quad\mbox{for}~\theta\neq\theta_0.
\label{eq:lim1}
\end{equation}
%where $K_i(\theta_0,\theta)>0$ for each $i$.
%Due to compactness of $\mathfrak T$ and $\mathfrak X$, where the sequences $\{T_1,T_2,\ldots,\}$ and $\{x^1,x^2,\ldots,\}$
%respectively belong, and also due to the compactness of $\Omega$ containing $\theta$, 
%$K_i(\theta_0,\theta)>c$, for some constant $c>0$. Hence, the right hand side
%of (\ref{eq:lim1}) is less that $-c$. 
In other words, (4)(i) is satisfied. 

To verify (4)(ii) we first show that 
$\underset{n\rightarrow\infty}{\lim}~\bar v_n$ exists for suitably chosen $r>0$. Then we prove that the limit
is negative. To see that the limit exists, we first write $R_{x,T}(\theta)=-\mathcal K_{x,T}(\theta_0,\theta)$.
%Since $K_{x,T}(\theta_0,\theta)$
Clearly, $R_i(\theta)=R_{x^i,T_i}(\theta)$. 
%For any $r>0$ such that $\left\{\theta:\|\theta\|>r\right\}\subset\Omega$, 
%it follows from (\ref{eq:upper_bound1}) that 
%$\underset{\left\{\theta:\|\theta\|>r\right\}}{\sup}~R_i(\theta)\leq \underset{\Omega}{\sup}~R_i(\theta)$
%for every $i\geq 1$. The right hand side is bounded by the same expression as the right hand side of (\ref{eq:sup_R1}),
%with only $S(\theta,\rho)$ replaced with $\Omega$. 
Using the arguments provided in the course of verification of (3)(ii), and the moment existence result of \ctn{Maud12}, 
%and Theorem \ref{theorem:continuity} asserting continuity of the moments with respect to $x$ and $T$, 
yield
\begin{equation}
\underset{i\geq 1}{\sup}~E_{\theta_0}\left[\underset{\left\{\theta\in\Omega:\|\theta\|>r\right\}}{\sup}~R_i(\theta)\right]^2
\leq
\underset{i\geq 1}{\sup}~E_{\theta_0}\left[\underset{\theta\in\Omega}{\sup}~R_i(\theta)\right]^2
\leq \underset{x\in \mathfrak X,T\in \mathfrak T}{\sup}~K_1(x,T),
\label{eq:upper_bound3}
\end{equation}
where $K_1(x,T)$ is a continuous function of $x$ and $T$. 
%the expectation of the right hand side of (\ref{eq:sup_R1}). 
That $K_1(x,T)$
is continuous in $x$ and $T$ follows from Theorem \ref{theorem:moment_continuity}; the required uniform
integrability follows
due to finiteness of the moments of the random variable $U(x,T)/\left\{1+\omega^2V(x,T)\right\}$, 
for every $x\in\mathfrak X$ and $T\in\mathfrak T$, 
and compactness of $\mathfrak X$ and $\mathfrak T$.
%asserting continuity of the moments 
%with respect to $x$ and $T$.  
Now, because of compactness of $\mathfrak X$ and $\mathfrak T$ it also follows that 
the right hand side of (\ref{eq:upper_bound3}) is finite, proving uniform integrability of 
$\left\{\underset{\left\{\theta\in\Omega:\|\theta\|>r\right\}}{\sup}~R_i(\theta)\right\}_{i=1}^{\infty}$. Hence, it follows 
from Theorem \ref{theorem:moment_continuity}
that $v_{x,T}=E_{\theta_0}\left[\underset{\left\{\theta\in\Omega:\|\theta\|>r\right\}}{\sup}~R_{x,T}(\theta)\right]$ 
is continuous in $x$ and $T$.
Since $x^i\rightarrow x^{\infty}$ and $T_i\rightarrow T^{\infty}$, 
$$\bar v_i=\bar v_{x^i,T_i}
\rightarrow \bar v_{x^{\infty},T^{\infty}}. %=\bar r(\theta),~\mbox{(say)},
$$
Since $\bar v_{x,T}$ is well-defined for every $x\in \mathfrak X$, $T\in \mathfrak T$, and since
$x^{\infty}\in \mathfrak X$, $T^{\infty}\in \mathfrak T$, $\bar v_{x^{\infty},T^{\infty}}$ is also well-defined. It follows that
\[
\underset{n\rightarrow\infty}{\lim}~\bar v_n = \bar v_{x^{\infty},T^{\infty}}
\]
exists.

To show that the limit $\underset{n\rightarrow\infty}{\lim}~\bar v_n$ is negative, 
let us first re-write ${\mathcal V}_i(r)$ as 
\begin{align}
{\mathcal V}_i(r)&=-\underset{\left\{\theta\in\Omega:\|\theta\|>r\right\}}{\inf}~
\left[\log\frac{f_i(X_i|\theta_0)}{f_i(X_i|\theta)}\right]\notag\\
&\leq -\underset{\left\{\theta\in\Omega:\|\theta\|\geq r\right\}}{\inf}~
\left[\log\frac{f_i(X_i|\theta_0)}{f_i(X_i|\theta)}\right]\notag\\
&=-\log\frac{f_i(X_i|\theta_0)}{f_i(X_i|\theta^*_i(X_i))},
\label{eq:vi_r}
\end{align}
for some $\theta^*_i(X_i)$, depending upon $X_i$, contained in $\Omega_r=\Omega\cap\left\{\theta:\|\theta\|\geq r\right\}$. 
Recall that we chose $r>0$ such that $\|\theta_0\|<r$ and 
$\Omega\cap\left\{\theta:\|\theta\|> r\right\}\neq\emptyset$, so that $\theta^*_i(X_i)\neq\theta_0$ as 
$\|\theta^*_i(X_i)\|\geq r>\|\theta_0\|$ for all $X_i$. 
It is important to observe that $\theta^*_i(X_i)$ can not be a one-to-one function of $X_i\equiv (U_i,V_i)$.
To see this, first observe that for any given constant $c$, the equation %it follows from the expression of 
$\log f_i(X_i|\theta_0)-\log f_i(X_i|\theta)=c$, equivalently, the equation
 $\log f_i(U_i,V_i|\theta_0)-\log f_i(U_i,V_i|\theta)=c$,
%(the expression of the left hand side is provided in the proof of Proposition 7 of \ctn{Maud12}) 
admits infinite number of solutions 
in $(U_i,V_i)$, for any given $\theta=(\mu,\omega^2)$. Hence, for  
$\theta^*_i(X_i)=\varphi$ such that $\underset{\left\{\theta:\|\theta\|\geq r\right\}}{\inf}~
\left[\log\frac{f_i(X_i|\theta_0)}{f_i(X_i|\theta)}\right]=\log\frac{f_i(X_i|\theta_0)}{f_i(X_i|\varphi)}=c$, 
there exist infinitely many values of $(U_i,V_i)$ with the same infimum $c$ for the same value $\varphi$, thereby
%$\log\frac{f_i(X_i|\theta_0)}{f_i(X_i|\varphi)}$, 
proving that $\theta_i^*(X_i)$
is a many-to-one function of $X_i$. A consequence of this is non-degeneracy of the conditional distribution
of $X_i$, given $\theta_i^*(X_i)$, which ensures that 
$E_{X_i|\theta^*_i(X_i),\theta_0}\left[\log\frac{f_i(X_i|\theta_0)}{f_i(X_i|\theta^*_i(X_i))}\right]
=\mathcal K_i(\theta_0,\theta^*_i(X_i))$
is well-defined and strictly positive, since $\theta^*_i(X_i)\neq\theta_0$.

Given the above arguments, now note that,
\begin{align}
E_{\theta_0}\left[\log\frac{f_i(X_i|\theta_0)}{f_i(X_i|\theta^*_i(X_i))}\right]
&=E_{\theta^*_i(X_i)|\theta_0}E_{X_i|\theta^*_i(X_i)=\varphi_i,\theta_0}
\left[\log\frac{f_i(X_i|\theta_0)}{f_i(X_i|\theta^*_i(X_i)=\varphi_i)}\right]\notag\\
&=E_{\theta^*_i(X_i)|\theta_0}\left[\mathcal K_i(\theta_0,\varphi_i)\right]\notag\\
&\geq E_{\theta^*_i(X_i)|\theta_0}\left[\underset{\varphi_i\in \Omega_r}{\inf}~\mathcal K_i(\theta_0,\varphi_i)\right]\notag\\
&=E_{\theta^*_i(X_i)|\theta_0}\left[\mathcal K_i(\theta_0,\varphi^*_i)\right]\notag\\
&=\mathcal K_i(\theta_0,\varphi^*_i),
\label{eq:kl_new1}
\end{align}
where $\varphi^*_i\in\Omega_r$ is where the infimum of $\mathcal K_i(\theta_0,\varphi_i)$ is achieved.
Since $\varphi^*_i$ is independent of $X_i$,  
the last step (\ref{eq:kl_new1}) follows.
Hence,
\begin{align}
E_{\theta_0}{\mathcal V}_i(r)&\leq -\mathcal K_i(\theta_0,\varphi^*_i)
\leq -\underset{x\in \mathfrak X,T\in \mathfrak T,\theta\in\Omega_r}{\inf}~\mathcal K_{x,T}(\theta_0,\varphi)
=-\mathcal K_{x^*,T^*}(\theta_0,\varphi^*),
\label{eq:kl_star}
\end{align}
for some $x^*\in \mathfrak X$, $T^*\in \mathfrak T$ and $\varphi^*\in\Omega_r$. Since $\mathcal K_{x^*,T^*}(\theta_0,\varphi^*)$
is a well-defined Kullback-Leibler distance, it is strictly positive since $\varphi^*\neq\theta_0$.
%$\mathcal N=\left\{\theta:\|\theta\|>r\right\}$. 
%
%Let us fix $\theta\neq\theta_0$ such that
%$\|\theta\|>r$ and that $\mathcal N^{(k)}_{\theta}$ is a closed ball centered at $\theta$ with radius
%less than or equal to $1/k$, contained wholly within $\mathcal N$. Let $Z(\mathcal N^{(k)}_{\theta},X_i)=
%\underset{\mathcal N^{(k)}_{\theta}}{\inf}\left[\log\frac{f_i(X_i|\theta_0)}{f_i(X_i|\theta)}\right]$.
%Clearly, 
%\begin{equation}
%Z(\mathcal N^{(k)}_{\theta},X_i)\geq Z(\mathcal N,X_i).
%\label{eq:z1}
%\end{equation}
%Using the same arguments as in the proof of Lemma 7.54 of \ctn{Schervish95} we conclude that
%if 
%\begin{equation}
%E_{\theta_0}Z(\mathcal N,X_i) >-\infty, 
%\label{eq:z2}
%\end{equation}
%then
%\begin{equation}
%\underset{k\rightarrow\infty}{\lim\inf}~E_{\theta_0}Z(\mathcal N^{(k)}_{\theta},X_i)\geq K_i(\theta_0,\theta).
%\label{eq:vi_r_2}
%\end{equation}
%That (\ref{eq:z2}) holds in our case is clear from the lower bound of $\log\frac{f_i(X_i|\theta_0)}{f_i(X_i|\theta)}$
%obtained by \ctn{Maud12} and their moment existence result.
%Hence, we can choose $k^*_i(\theta)$ such that $E_{\theta_0}Z(\mathcal N^{(k^*_i)}_{\theta},X_i)\geq K_i(\theta_0,\theta)$.
%Also, due to (\ref{eq:z1}), ${\mathcal V}_i(r)\leq -Z(\mathcal N^{(k^*_i)}_{\theta},X_i)$, 
%so that using (\ref{eq:vi_r_2}) yields
%
Hence, it follows from (\ref{eq:kl_star}) and the fact that $\mathcal K_{x^*,T^*}(\theta_0,\varphi^*)>0$, that
\begin{align}
\underset{n\rightarrow\infty}{\lim}~\bar v_n
&=\underset{n\rightarrow\infty}{\lim}~\frac{\sum_{i=1}^n E_{\theta_0}{\mathcal V}_i(r)}{n}
\notag\\
%&\leq -\underset{n\rightarrow\infty}{\lim}\frac{\sum_{i=1}^n E_{\theta_0}Z(\mathcal N^{(k^*_i)}_{\theta},X_i)}{n}
%\notag\\
&\leq -\underset{n\rightarrow\infty}{\lim}~\frac{\sum_{i=1}^n \mathcal K_i(\theta_0,\varphi^*_i)}{n}\notag\\
&\leq -\mathcal K_{x^*,T^*}(\theta_0,\varphi^*)\notag\\
&<0.\notag
\end{align}
%the last step following due to (\ref{eq:kl_limit_1}).
Thus, condition (4)(ii) holds.

Regularity condition (5) holds because for any $\theta\in\Omega$, $R_i(\theta)$ is an almost surely
continuous function of $X_i$ rendering it measurable for all
$\theta\in\Omega$, and due to the fact that supremums of measurable functions
are measurable.

In other words, in the non-$iid$ set-up in the non-$iid$ SDE framework, the following
theorem holds:
\begin{theorem}
\label{theorem:consistency_non_iid}
Assume the non-$iid$ SDE setup and conditions (H1$^\prime$) (i) and (H2$^\prime$) -- (H6$^\prime$). 
Then it holds that $\hat\theta_n\stackrel{P}{\rightarrow}\theta_0$.
\end{theorem}

\subsection{Asymptotic normality of $MLE$ in the non-$iid$ set-up}
\label{subsec:normality_non_iid}

Let $\zeta_i(x,\theta)=\log f_i(x|\theta)$; also, let $\zeta'_i(x,\theta)$ be the $d\times 1$ vector
with $j$-th component $\zeta'_{i,j}(x,\theta)=\frac{\partial}{\partial\theta_j}\zeta_i(x,\theta)$, and
let $\zeta''_i(x,\theta)$ be the $d\times d$ matrix with $(j,k)$-th element
$\zeta''_{i,jk}(x,\theta)=\frac{\partial^2}{\partial\theta_j\partial\theta_k}\zeta_i(x,\theta)$.

For proving asymptotic normality in the non-$iid$ framework, \ctn{Hoadley71} assumed the
following regularity conditions:
\begin{itemize}
\item[(1)] $\Omega$ is an open subset of $\mathcal R^d$.
\item[(2)] $\hat\theta_n\stackrel{P}{\rightarrow}\theta_0$.
\item[(3)] $\zeta'_i(X_i,\theta)$ and $\zeta''_i(X_i,\theta)$ exist a.s. $[P_{\theta_0}]$.
\item[(4)] $\zeta''_i(X_i,\theta)$ is a continuous function of $\theta$, uniformly in $i$, a.s. $[P_{\theta_0}]$,
and is a measurable function of $X_i$.
\item[(5)] $E_{\theta}[\zeta'_i(X_i,\theta)]=0$ for $i=1,2,\ldots$.
\item[(6)] $\mathcal I_i(\theta)=E_{\theta}\left[\zeta'_i(X_i,\theta)\zeta'_i(X_i,\theta)^T\right]
=-E_{\theta}\left[\zeta''_i(X_i,\theta)\right]$, where for any vector $y$, $y^T$ denotes
the transpose of $y$.
\item[(7)] $\bar{\mathcal I}_n(\theta)\rightarrow\bar{\mathcal I}(\theta)$ as $n\rightarrow\infty$ and 
$\bar{\mathcal I}(\theta)$ is positive definite.
\item[(8)] $E_{\theta_0}\left\vert\zeta'_{i,j}(X_i,\theta_0)\right\vert^3\leq K_2$, for some $0<K_2<\infty$.
\item[(9)] There exist $\epsilon>0$ and random variables $B_{i,jk}(X_i)$ such that
\begin{enumerate}
\item[(i)] $\sup\left\{\left\vert\zeta''_{i,jk}(X_i,\xi)\right\vert:\|\xi-\theta_0\|\leq\epsilon\right\}
\leq B_{i,jk}(X_i)$.
\item[(ii)] $E_{\theta_0}\left\vert B_{i,jk}(X_i)\right\vert^{1+\delta}\leq K_2$, for some $\delta>0$.
\end{enumerate}
\end{itemize}
Condition (8) can be weakened but is relatively easy to handle.
Under the above regularity conditions, \ctn{Hoadley71} prove that 
\begin{equation}
\sqrt{n}\left(\hat\theta_n-\theta_0\right)\stackrel{\mathcal L}{\rightarrow}
N\left(\bzero,\bar{\mathcal I}^{-1}(\theta_0)\right).
\label{eq:MLE_normality_non_iid}
\end{equation}

\subsubsection{Validation of asymptotic normality of $MLE$ in the non-$iid$ SDE set-up}
\label{subsubsec:normality_non_iid}

Note that condition (1) requires the parameter space $\Omega$ to be an open subset. However,
the proof of asymptotic normality presented in \ctn{Hoadley71} continues to hold for compact $\Omega$,
since for any open cover of $\Omega$ we can extract a finite subcover, consisting of open sets. 
%Thus, condition (1)
%holds. 

Conditions (2), (3), (5), (6) are clearly valid in our case.
Condition (4) can be verified in exactly the same way as condition (2) of Section \ref{subsec:consistency_non_iid}
is verified; measurability of $\zeta''_i(X_i,\theta)$ follows due its continuity with respect to $X_i$.
%
%For condition (7) note that
%$\bar{\mathcal I}_n(\theta)$, being an average of positive definite matrices, is positive definite
%for every $n$. Moreover, because of compactness of $\mathfrak T$ and $\mathfrak X$ which contain the sequences
%$\{T_1,T_2,\ldots\}$ and $\{x^1,x^2,\ldots,\}$ respectively, the limit 
%$\underset{n\rightarrow\infty}{\lim}\bar{\mathcal I}_n(\theta)$ is non-zero, and hence positive definite.
%%That the limit exists is clear since the sequences $\{T_1,T_2,\ldots\}$ and $\{x^1,x^2,\ldots,\}$ are aperiodic;
%Here $\bar{\mathcal I}(\theta)$ stands for the limit.
%Hence, condition (7) holds.
%
Condition (7) simply follows from (\ref{eq:fisher_limit_1}) and
condition (8) holds due to finiteness of the moments of the random variable $U(x,T)/\left\{1+\omega^2V(x,T)\right\}$, 
 for every $x\in\mathfrak X$, $T\in\mathfrak T$,
and compactness of $\mathfrak X$ and $\mathfrak T$.

For conditions (9)(i) and (9)(ii) note that $\zeta''_{i,jk}(X_i,\theta)$ for 
$j,k=1,2$ are given by $\frac{\partial^2}{\partial\mu^2}\log f(X_i|\theta)=-I_i(\omega^2)$,
$\frac{\partial^2}{\partial\mu\partial\omega^2}\log f(X_i|\theta)=-\gamma_i(\theta)I_i(\omega^2)$,
and $\frac{\partial^2}{\partial\omega^2\partial\omega^2}\log f(X_i|\theta)
=-\frac{1}{2}\left(2\gamma^2_i(\theta)I_i(\omega^2)-I^2_i(\omega^2)\right)$. 
Also since \ctn{Maud12} establish
\begin{equation}
\underset{\theta\in\Omega}{\sup}~|\gamma_i(\theta)|\leq\left\vert\frac{U_i}{1+\omega^2_0V_i}\right|
\left(2+\frac{\omega^2_0}{\underline\omega^2}\right)+\frac{|\bar\mu |}{\underline\omega^2},
\label{eq:cond1_3}
\end{equation}
it follows from
(\ref{eq:cond1_3}), the fact that $0<I_i(\omega^2)<1/\omega^2$, 
finiteness of moments of all orders of the previously mentioned derivatives for every $x\in\mathfrak X$, $T\in\mathfrak T$, 
and compactness of $\mathfrak X$ and $\mathfrak T$,
that conditions (9)(i) and (9)(ii) hold.

In other words, in our non-$iid$ SDE case we have the following theorem on asymptotic normality.
%of the form
%(\ref{eq:MLE_normality_non_iid}).
\begin{theorem}
\label{theorem:asymp_normal_non_iid}
Assume the non-$iid$ SDE setup and conditions (H1$^\prime$) (i) and (H2$^\prime$) -- (H6$^\prime$). 
Then (\ref{eq:MLE_normality_non_iid}) holds. 
\end{theorem}

\section{Summary and conclusion}
\label{sec:conclusion}

In $SDE$ based random effects model framework, \ctn{Maud12} considered 
the linearity assumption in the drift function given by $b(x,\phi_i) = \phi_ib(x)$, where $\phi_i$ are supposed to 
be Gaussian random variables with mean $\mu$ and variance $\omega^2$, and obtained a closed form 
expression of the likelihood of the above parameters. Assuming the $iid$ set-up, they proved convergence in 
probability and asymptotic normality of the maximum likelihood estimator of the parameters. 
In this paper, we proved strong consistency, rather than weak consistency, and asymptotic normality of the maximum
likelihood estimator under weaker assumptions in the $iid$ set-up. 
Moreover, we extended the model of \ctn{Maud12} to the independent, but non-identical set-up, proving
weak consistency and asymptotic normality.

In \ctn{Maitra14b}, we extended our classical asymptotic theory to the Bayesian framework, for both
$iid$ and non-$iid$ situations. Specifically, we proved posterior consistency and asymptotic posterior
normality, for both $iid$ and non-$iid$ set-ups. There we have also illustrated our theoretical development
with several examples and simulation studies. It is to be noted that those examples, illustrating 
consistency and inconsistency of the associated Bayes estimators, remains valid in the classical paradigm
with the Bayes estimators replaced by the maximum likelihood estimators.

\section*{Acknowledgments}
Sincere gratitude goes to the reviewer whose suggestions have led to
much improved presentation of our article.
The first author gratefully acknowledges her CSIR Fellowship, Govt. of India.

\section*{Appendix}

%\begin{appendix}

\begin{proof}[Proof of Theorem \ref{theorem:moment_continuity}]
We can decompose (\ref{eq:u_x_T}) as
\begin{align}
 U(x,T)&=\int_0^T\frac{b(X(s,x))}{\sigma^2(X(s,x))}\phi(x)b(X(s,x))ds\notag\\
&\quad + \int_0^T\frac{b(X(s,x))}{\sigma^2(X(s,x))}(d X(s,x)-\phi(x)b(X(s,x))ds)\notag\\
&=\phi(x)\int_0^T\frac{b^2(X(s,x))}{\sigma^2(X(s,x))}ds
\label{eq:U1}\\
&\quad +\int_0^T\frac{b(X(s,x))}{\sigma(X(s,x))}dW(s)
\label{eq:U2}\\
&=\phi(x)U^{(1)}(x,T)+U^{(2)}(x,T),\quad\mbox{(say)},
\label{eq:U1U2}
\end{align}
where $W(s)$ is the standard Weiner process defined on $[0,T]$. 
%Thus,
%\begin{align}
%U(x,T)&=\phi(x)U^{(1)}(x,T)+U^{(2)}(x,T).
%\label{eq:phi_U}
%\end{align}

Given the process $X(\cdot,\cdot)$, continuity of (\ref{eq:U1}) with respect to $x$ and $T$ 
can be seen as follows. Let $T_1,T_2\in \mathfrak T$; without loss of generality, let $T_2>T_1$. Also, let
$x_1,x_2\in \mathfrak X$. Then,
\begin{align}
&\left|U^{(1)}(x_1,T_1)-U^{(1)}(x_2,T_2)\right|\notag\\
&=\left|\int_0^{T_1}\frac{b^2(X(s,x_1))}{\sigma^2(X(s,x_1))}ds
-\int_0^{T_2}\frac{b^2(X(s,x_2))}{\sigma^2(X(s,x_2))}ds\right|\notag\\
&=\left|\int_0^{T_1}\left[\frac{b^2(X(s,x_1))}{\sigma^2(X(s,x_1))}
-\frac{b^2(X(s,x_2))}{\sigma^2(X(s,x_2))}\right]ds\right.\notag\\
&\quad\quad\quad\quad\left.-\int_{T_1}^{T_2}\frac{b^2(X(s,x_2))}{\sigma^2(X(s,x_2))}ds\right|\notag\\
&\leq\int_0^{T_1}\left|\frac{b^2(X(s,x_1))}{\sigma^2(X(s,x_1))}
-\frac{b^2(X(s,x_2))}{\sigma^2(X(s,x_2))}\right|ds\notag\\
&\quad\quad\quad\quad+\int_{T_1}^{T_2}\left|\frac{b^2(X(s,x_2))}{\sigma^2(X(s,x_2))}\right|ds\notag\\
&\leq T_1\underset{s\in [0,T_1]}{\sup}~\left|\frac{b^2(X(s,x_1))}{\sigma^2(X(s,x_1))}
-\frac{b^2(X(s,x_2))}{\sigma^2(X(s,x_2))}\right|\label{eq:U1_cont1}\\
&\quad\quad\quad\quad+|T_2-T_1|\underset{s\in [T_1,T_2],x\in \mathfrak X}{\sup}~\left|\frac{b^2(X(s,x))}{\sigma^2(X(s,x))}\right|
\label{eq:U2_cont1}\\
&\leq T_{\max}\left|\frac{b^2(X(s^{*},x_1))}{\sigma^2(X(s^{*},x_1))}
-\frac{b^2(X(s^{*},x_2))}{\sigma^2(X(s^{*},x_2))}\right|
+C_2|T_2-T_1|,%\left|\frac{b(X(s^{(2)},x_2))}{\sigma^2(X(s^{(2)},x_2))}b(X(s^{(2)},x_2))\right|,
\label{eq:U2_cont2}
%&= T_1\left|\frac{b(X(s^{(1)},x_1))}{\sigma^2(X(s^{(1)},x_1))}b(X(s^{(1)},x_1))
%-\frac{b(X(s^{(1)},x_2))}{\sigma^2(X(s^{(1)},x_2))}b(X(s^{(1)},x_2))\right|
%\label{eq:U1_cont3}\\
%&\quad\quad\quad\quad+C_2|T_2-T_1|\left|\frac{b(X(s^{(2)},x_2))}{\sigma^2(X(s^{(2)},x_2))}b(X(s^{(2)},x_2))\right|,
%\label{eq:U2_cont3}
\end{align}
where $T_{\max}=\sup~\mathfrak T$; $s^{*}\in [0,T_1]$ is such that the supremum in (\ref{eq:U1_cont1}) is
attained. That there exists such $s^{*}$ is clear due to continuity of the functions in $s$ and 
compactness of the interval. In (\ref{eq:U2_cont2}), $C_2$ is the upper bound for the function 
$\left|\frac{b^2(X(s,x))}{\sigma^2(X(s,x))}\right|$.

Since $X(\cdot,x)$ is continuous in $x$, due to continuity of $b(\cdot)$ and $\sigma(\cdot)$, for any $\epsilon>0$, 
one can choose $\delta_1(\epsilon)>0$ such that $|x_1-x_2|<\delta_1(\epsilon)$ implies 
$$\left|\frac{b^2(X(s^{*},x_1))}{\sigma^2(X(s^{*},x_1))}
-\frac{b^2(X(s^{*},x_2))}{\sigma^2(X(s^{*},x_2))}\right|<\frac{\epsilon}{2T_{\max}},$$
so that the first term in (\ref{eq:U2_cont2}) is less than $\epsilon/2$.
Choosing $\delta_2(\epsilon)=\frac{\epsilon}{2C_2}$ yields that if $|T_2-T_1|<\delta_2(\epsilon)$, then
the second term in (\ref{eq:U2_cont2}) is less than $\epsilon/2$.
This shows continuity of $U^{(1)}(x,T)$ with respect to $x$ and $T$ for given $X(\cdot,\cdot)$.
%almost all $e$ in the sample
%space of the random variable $U^{(1)}(x,T)$.
It follows that, for sequences $\{x_m\}_{m=1}^{\infty}$, $\{T_m\}_{m=1}^{\infty}$ such that $x_m\rightarrow \tilde x$ and
$T_m\rightarrow\tilde T$ as $m\rightarrow\infty$, 
\begin{equation}
U^{(1)}(x_m,T_m)\stackrel{\mathcal L}{\rightarrow}U^{(1)}(\tilde x,\tilde T).
\label{eq:U1_conv1}
\end{equation}
It is also clear that 
\begin{equation}
\phi(x_m)\stackrel{\mathcal L}{\rightarrow}\phi(\tilde x).
\label{eq:phi_conv1}
\end{equation}
Now note that due to assumptions (H5$^\prime$), %(\ref{eq:phi_sup1}), 
(H6$^\prime$) %(that is, (\ref{eq:V_finite_moment2})) 
(observing that
$U^{(1)}(x,T)=V(x,T)$ for all $x\in\mathfrak X$ and $T\in\mathfrak T$),
and compactness of $\mathfrak X$ and $\mathfrak T$, %(H2) of \ctn{Maud12}, 
%which postulates almost sure finiteness
%of $U^{(1)}(x,T)$ for any $x$ and $T$, and the assumption that 
%$\underset{x\in \mathfrak X}{\sup}~E\left[\phi(x)\right]^{2k}<\infty$ 
%for all $k\geq 1$,
we have, %for any continuous function $h(\cdot)$, for any $\epsilon>0$, 
for any $k\geq 1$, 
\begin{align}
\underset{m\geq 1}{\sup}~E\left[\phi(x_m)U^{(1)}(x_m,T_m)\right]^{2k}&<\infty,
%
%\leq K_2(x_m,T_m)
%\leq\underset{x\in \mathfrak X,T\in \mathfrak T}{\sup}~K_2(x,T)\leq K^*_2<\infty,
\label{eq:ui1}
\end{align}
for all $m\geq 1$, ensuring requisite uniform integrability. %of $\left\{\phi(x_m)U^{(1)}(x_m,T_m)\right\}_{m=1}^{\infty}$. 
%
%\begin{align}
%E\left|h\left(U^{(1)}(x_m,T_m)\right)\right|^{2+\epsilon}&\leq K_2(x_m,T_m)\leq\underset{x\in \mathfrak X,T\in \mathfrak T}{\sup}K_2(x,T)\leq K^*_2<\infty,
%\label{eq:ui1}
%\end{align}
%so that $\underset{m}{\sup}~E\left|h\left(U^{(1)}(x_m,T_m)\right)\right|^{2+\epsilon}<\infty$.
%In other words, the sequence
%$\left\{\left|h\left(U^{(1)}(x_m,T_m)\right)\right|^2\right\}_{m=1}^{\infty}$
%is uniformly integrable. 
Hence, it follows that
%\begin{align}
%E\left[U^{(1)}(x_m,T_m)\right]^k\rightarrow E\left[U^{(1)}(\tilde x,\tilde T)\right]^k.
%\label{eq:ui2}
%\end{align}
\begin{align}
%E\left[h\left(U^{(1)}(x_m,T_m)\right)\right]&\rightarrow E\left[h\left(U^{(1)}(\tilde x,\tilde T)\right)\right],
E\left[\phi(x_m)U^{(1)}(x_m,T_m)-\phi(\tilde x)U^{(1)}(\tilde x,\tilde T)\right]^2\rightarrow 0.
\label{eq:ui2}
\end{align}
%implying that $E\left[h\left(U^{(1)}(x,T)\right)\right]$ is continuous with respect to
%$x$ and $T$.
%Assuming that $\underset{x\in \mathfrak X,T\in \mathfrak T}{\sup}~E\left|\phi(x,T)\right|^{2k}<\infty$ for all $k\geq 1$, it can be similarly shown
%that $E\left[\phi(x,T)\right]$ is continuous in $x$ and $T$. Since, due to independence,
%$$E\left[\phi(x,T)U^{(1)}(x,T)\right]=E\left[\phi(x,T)\right]E\left[U^{(1)}(x,T)\right],$$ 
%it follows that $E\left[\phi(x,T)h(U^{(1)}(x,T))\right]$ is continuous in $x$ and $T$.

Let us now deal with $U^{(2)}(x,T)$ given by (\ref{eq:U2}).
%For this purpose let us consider the sequences $\{x_m\}_{m=1}^{\infty}$ and $\{T_m\}_{m=1}^{\infty}$
%such that $x_m\rightarrow\tilde x$ and $T_m\rightarrow\tilde T$ as $m\rightarrow\infty$. 
Letting for any set $A$, $\delta_A(s)=1$ if $s\in A$ and $0$ otherwise, be the indicator function,
we define
\begin{align}
Q(x_m,T_m)&=\int_0^{T_{\max}}\left[\frac{b(X(s,x_m))}{\sigma(X(s,x_m))}\delta_{[0,T_m]}(s)-
\frac{b(X(s,\tilde x))}{\sigma(X(s,\tilde x))}\delta_{[0,\tilde T]}(s)\right]^2ds\notag\\
&=\int_0^{T_{\max}}\frac{b^2(X(s,x_m))}{\sigma^2(X(s,x_m))}\delta_{[0,T_m]}(s)ds+
\int_0^{T_{\max}}\frac{b^2(X(s,\tilde x))}{\sigma^2(X(s,\tilde x))}\delta_{[0,\tilde T]}(s)ds\notag\\
&\quad\quad-2\int_0^{T_{\max}}\frac{b(X(s,x_m))}{\sigma(X(s,x_m))}
\frac{b(X(s,\tilde x))}{\sigma(X(s,\tilde x))}\delta_{[0,\min\{T_m,\tilde T\}]}(s)ds\notag\\
&=\int_0^{T_{m}}\frac{b^2(X(s,x_m))}{\sigma^2(X(s,x_m))}ds
+\int_0^{\tilde T}\frac{b^2(X(s,\tilde x))}{\sigma^2(X(s,\tilde x))}ds\notag\\
&\quad\quad-2\int_0^{\min\{T_m,\tilde T\}}\frac{b(X(s,x_m))}{\sigma(X(s,x_m))}
\frac{b(X(s,\tilde x))}{\sigma(X(s,\tilde x))}ds.
\label{eq:Q1}
\end{align}

%where, for any set $A$, $\delta_A(s)=1$ if $s\in A$ and $0$ otherwise, is the indicator function.
%Given $X(\cdot,\cdot)$, continuity of the integrals in (\ref{eq:Q1})
%with respect to $x_m$ implies pointwise convergence
%of the integrand of $Q(x_m,T_m)$ to zero. Dominated convergence theorem exploiting uniform boundedness
%of the continuous function $b(X(\cdot,\cdot))$ on $[0,{T_{\max}}]$ and the indicator functions ensure 
It follows in the same way as in the proof of continuity of $U^{(1)}(\cdot,\cdot)$ that 
the first and the third integrals in (\ref{eq:Q1}) associated with the function $Q(\cdot,\cdot)$, are 
continuous at $(\tilde x,\tilde T)$. As a result, for given $X(\cdot,\cdot)$,
$Q(x_m,T_m)\rightarrow 0$ as $m\rightarrow\infty$. It follows that
$Q(x_m,T_m)\stackrel{\mathcal L}{\rightarrow} 0$. 

Now note that 
$$Q(x_m,T_m)\leq 2\left[\left(\int_0^{T_{\max}}\frac{b^2(X(s,x_m))}{\sigma^2(X(s,x_m))}ds\right)^2
+\left(\int_0^{T_{\max}}\frac{b^2(X(s,\tilde x))}{\sigma^2(X(s,\tilde x))}ds\right)^2\right],$$
so that, for any $k\geq 2$,
\begin{align}
E\left[Q(x_m,T_m)\right]^k&\leq 
2^{2k}E\left[\left(\int_0^{T_{\max}}\frac{b^2(X(s,x_m))}{\sigma^2(X(s,x_m))}ds\right)^{2k}
+\left(\int_0^{T_{\max}}\frac{b^2(X(s,\tilde x))}{\sigma^2(X(s,\tilde x))}ds\right)^{2k}\right].%\notag\\
%&\leq
%K_4E\left[\left(\int_0^{T_{\max}}\left(1+X(s,x_m)^2\right)ds\right)^{2k}
%+\left(\int_0^{T_{\max}}(1+X(s,\tilde x)^2)ds\right)^{2k}\right]\notag\\
%&\leq K_5\int_0^{T_{\max}}E\left[X(s,x_m)^2\right]ds+K_6\int_0^{T_{\max}}E\left[X(s,\tilde x)^2\right]ds\notag\\
%&\leq K_5{T_{\max}}\underset{t\in [0,T_{\max}]}{\sup}~E\left[X(t,x_m)^2\right]
%+K_6T_{\max}\underset{t\in [0,T_{\max}]}{\sup}~E\left[X(t,\tilde x)^2\right]\notag\\
%&\leq K_7{T_{\max}}\underset{t\in [0,T_{\max}],x\in \mathfrak X}{\sup}~E\left[X(t,x)^2\right]\notag\\
%&<\infty,\quad\mbox{[due to (\ref{eq:sup_X})]}
\label{eq:Q3}
\end{align}
Since, by assumption (H6$^\prime$) %(\ref{eq:V_finite_moment2}), 
moments of all orders of
$V(x,T)$ are finite,
%$\int_0^{T}\frac{b^2(X(s,x))}{\sigma^2(X(s,x))}ds<\infty$ 
for any $x\in\mathfrak X$ and $T\in\mathfrak T$, and since $\mathfrak X$ and $\mathfrak T$ are compact,
it follows that 
\[
\underset{m\geq 1}{\sup}~ E\left[Q(x_m,T_m)\right]^k<\infty,
\]
guaranteeing uniform integrabiility.
%showing that the sequence $\left\{\left|Q(x_m,T_m)\right|^{k-1}\right\}_{m=1}^{\infty}$ is uniformly integrable for $k\geq 2$.
%In particular, choosing $k=3$ serves our purpose.
Hence,
\begin{equation}
E\left[Q(x_m,T_m)\right]\rightarrow 0,\quad\mbox{as}~m\rightarrow\infty.
\end{equation}
By It\^{o} isometry (see, for example, \ctn{Oksendal03}) it then follows that
\begin{equation}
E\left[\int_0^{T_{\max}}\frac{b(X(s,x_m))}{\sigma(X(s,x_m))}\delta_{[0,T_m]}(s)dW(s)
-\int_0^{T_{\max}}\frac{b(X(s,\tilde x))}{\sigma(X(s,\tilde x))}\delta_{[0,\tilde T]}(s)dW(s)\right]^2
\rightarrow 0.
\label{eq:isometry1}
\end{equation}
That is,
\begin{equation}
E\left[\int_0^{T_m}\frac{b(X(s,x_m))}{\sigma(X(s,x_m))}dW(s)
-\int_0^{\tilde T}\frac{b(X(s,\tilde x))}{\sigma(X(s,\tilde x))}dW(s)\right]^2
\rightarrow 0.
\label{eq:isometry2}
\end{equation}
It follows that
\begin{equation}
U^{(2)}(x_m,T_m)\stackrel{\mathcal L}{\rightarrow}U^{(2)}(\tilde x,\tilde T).
\label{eq:U2_conv1}
\end{equation}
Using the Burkholder-Davis-Gundy inequality (see, for example, \ctn{Maud12}) we obtain
\begin{align}
E\left[U^{(2)}(x_m,T_m)\right]^{2k}&\leq C_kE\left[\int_0^{T_m}\frac{b^2(X(s,x_m))}{\sigma^2(X(s,x_m))}ds\right]^k.%\notag\\
%&\leq C_kE\left[\int_0^{T_m}K\left(1+\left|X(s,x_m)\right|^2\right)ds\right]^k.
\label{eq:ui5}
\end{align}
Again, due to assumption (H6$^\prime$) %(\ref{eq:V_finite_moment2}) 
and compactness of $\mathfrak X$ and $\mathfrak T$ it follows that
%since because of (\ref{eq:sup_X}),
%$\underset{s\in [0,T_{\max}],x\in \mathfrak X}{\sup}~E\left[X(s,x)\right]^{2k}<\infty$ for all $k\geq 1$,
%it is easy to see that 
$\underset{m\geq 1}{\sup}~E\left[U^{(2)}(x_m,T_m)\right]^{2k}<\infty$, so that uniform integrability is assured.
%the sequence $\left\{\left|U^{(2)}(x_m,T_m)\right|^{2k-1}\right\}_{m=1}^{\infty}$ is uniformly integrable, for $k\geq 1$.
It follows that
\begin{align}
%E\left[U^{(2)}(x_m,T_m)\right]^k&\rightarrow E\left[U^{(2)}(\tilde x,\tilde T)\right]^k,
E\left[U^{(2)}(x_m,T_m)-U^{(2)}(\tilde x,\tilde T)\right]^2\rightarrow 0.
\label{eq:ui3}
\end{align}

From (\ref{eq:ui2}) and (\ref{eq:ui3}) it follows, using the Cauchy-Schwartz inequality, that 
\begin{align}
&E\left[U(x_m,T_m)-U(\tilde x,\tilde T)\right]^2\notag\\
&\leq E\left[\phi(x_m)U^{(1)}(x_m,T_m)-\phi(\tilde x)U^{(1)}(\tilde x,\tilde T)\right]^2
+E\left[U^{(2)}(x_m,T_m)-U^{(2)}(\tilde x,\tilde T)\right]^2\notag\\
&\quad\quad+2\sqrt{E\left[\phi(x_m)U^{(1)}(x_m,T_m)-\phi(\tilde x)U^{(1)}(\tilde x,\tilde T)\right]^2
E\left[U^{(2)}(x_m,T_m)-U^{(2)}(\tilde x,\tilde T)\right]^2}\notag\\
&\rightarrow 0.
\label{eq:phi_U3}
\end{align}
Since $V(x,T)=U^{(1)}(x,T)$, due to %(\ref{eq:ui2}) we must have
(\ref{eq:U1_conv1}) and assumption (H6$^\prime$) %(\ref{eq:V_finite_moment2}) 
(the latter ensuring uniform integrability), it easily follows
that
\begin{align}
%&E\left[\phi(x_m)V(x_m,T_m)-\phi(\tilde x)V(\tilde x,\tilde T)\right]^2
&E\left[V(x_m,T_m)-V(\tilde x,\tilde T)\right]^2
\rightarrow 0.
\label{eq:phi_V}
\end{align}
Let $G(x,T)=\left(U(x,T), V(x,T)\right)$. Then it follows from
(\ref{eq:phi_U3}) and (\ref{eq:phi_V}), that 
\begin{equation}
G(x_m,T_m)\stackrel{\mathcal L}{\rightarrow}G(\tilde x,\tilde T).
\label{eq:G1}
\end{equation}
%Letting $H(x,T)=G(x,T)/\phi(x,T)=\left(U(x,T), V(x,T)\right)^T$, 
%and using the continuous mapping theorem it follows that
%\begin{equation}
%\frac{G(x_m,T_m)}{\phi(x_m)}\stackrel{\mathcal L}{\rightarrow}\frac{G(\tilde x,\tilde T)}{\phi(\tilde x)},
%\label{eq:G2}
%\end{equation}
That is,
\begin{equation}
\left(U(x_m,T_m), V(x_m,T_m)\right)\stackrel{\mathcal L}{\rightarrow}\left(U(\tilde x,\tilde T), V(\tilde x,\tilde T)\right).
\label{eq:UV}
\end{equation}
Consider also a sequence $\left\{\theta_m\right\}_{m=1}^{\infty}$ in $\Omega$, converging to $\tilde\theta\in\Omega$.
Then, for any function $h(u,v,\theta)$, which is continuous in $u$, $v$ and $\theta$, and such that 
the sequence $\left\{h(U(x_m,T_m),V(x_m,T_m),\theta_m)\right\}_{m=1}^{\infty}$ is uniformly integrable, we must have 
\begin{equation}
E\left[h(U(x_m,T_m),V(x_m,T_m),\theta_m)\right]\rightarrow E\left[h(U(\tilde x,\tilde T),V(\tilde x,\tilde T),\tilde\theta)\right],
\label{eq:h_H}
\end{equation}
ensuring continuity of $ E\left[h(U(x,T),V(x,T),\theta)\right]$ with respect to $x$, $T$ and $\theta$.
\end{proof}

%\end{appendix}

%\newpage

%\renewcommand\baselinestretch{1.3}
%\renewcommand\baselinestretch{1}
\normalsize
\bibliographystyle{natbib}
\bibliography{irmcmc}

\end{document}